\documentclass[12pt]{amsart}

\usepackage{amsmath, amssymb, amsthm, amsfonts}

\input xy
\xyoption{all}

\usepackage{hyperref}

\swapnumbers
\numberwithin{equation}{section}

\theoremstyle{plain}
\newtheorem{theorem}[subsection]{Theorem}
\newtheorem{lemma}[subsection]{Lemma}
\newtheorem{prop}[subsection]{Proposition}
\newtheorem{cor}[subsection]{Corollary}

\newtheorem{theorem-conjecture}[subsection]{Theorem-Conjecture}

\theoremstyle{definition}

\newtheorem{remark}[subsection]{Remark}
\newtheorem{exam}[subsection]{Example}

\def\AA{\mathbb{A}}

\def\CC{\mathbb{C}}

\def\FF{\mathbb{F}}
\def\GG{\mathbb{G}}

\def\PP{\mathbb{P}}
\def\QQ{\mathbb{Q}}
\def\RR{\mathbb{R}}

\def\ZZ{\mathbb{Z}}

\def\calC{\mathcal{C}}

\def\calF{\mathcal{F}}

\def\calO{\mathcal{O}}

\newcommand\frg{\mathfrak{g}}

\newcommand\frd{\mathfrak{d}}

\newcommand\frX{\mathfrak{X}}

\newcommand\tilC{\widetilde{C}}

\newcommand\Frob{\textup{Frob}}

\newcommand{\Gr}{\textup{Gr}}

\newcommand{\Hilb}{\textup{Hilb}}

\newcommand\id{\textup{id}}

\newcommand\Jac{\textup{Jac}}

\newcommand{\Pic}{\textup{Pic}}

\newcommand{\Quot}{\textup{Quot}}

\newcommand{\red}{\textup{red}}

\newcommand{\Res}{\textup{Res}}

\newcommand\Spec{\textup{Spec}}

\newcommand\Stab{\textup{Stab}}

\newcommand{\Tr}{\textup{Tr}}

\newcommand{\val}{\textup{val}}
\newcommand{\vol}{\textup{vol}}

\newcommand\Aut{\textup{Aut}}
\newcommand\Hom{\textup{Hom}}
\newcommand\End{\textup{End}}

\newcommand\GL{\textup{GL}}
\newcommand\gl{\mathfrak{gl}}

\newcommand{\Gm}{\GG_m}

\newcommand{\incl}{\hookrightarrow}
\newcommand{\isom}{\stackrel{\sim}{\to}}

\newcommand{\jiao}[1]{\langle{#1}\rangle}

\newcommand{\wh}[1]{\widehat{#1}}
\newcommand{\one}{\mathbf{1}}
\newcommand{\ep}{\epsilon}
\newcommand\quash[1]{}
\newcommand{\lc}{{\em loc. cit.}}

\newcommand{\cohog}[2]{\textup{H}^{#1}({#2})}     

\newcommand{\kbar}{\overline{k}}
\newcommand\tilk{\widetilde{k}}

\newcommand\tilJ{\widetilde{J}}

\renewcommand{\l}{\lambda}
\renewcommand{\L}{\Lambda}

\newcommand{\leng}{\textup{leng}}

\newcommand{\Cl}{\textup{Cl}}
\newcommand{\bCl}{\overline{\Cl}}
\newcommand{\cJac}{\overline{\Jac}}
\newcommand{\cPic}{\overline{\Pic}}

\title{Orbital integrals and Dedekind zeta functions}

\dedicatory{Dedicated to Srinivasa Ramanujan's 125th Birthday}
\author{Zhiwei Yun}
\thanks{Supported by the NSF grants DMS-0969470 and DMS-1261660.}
\address{Department of Mathematics, Stanford University, 450 Serra Mall, Building 380, Stanford, CA 94305}
\email{zwyun@stanford.edu}
\date{}
\subjclass[2010]{Primary 22E35; Secondary 11R54}
\keywords{Zeta functions; Orbital integrals}

\begin{document}

\begin{abstract}
Let $F$ be a local non-archimedean field. We prove a formula relating orbital integrals in $\GL(n,F)$ (for the unit Hecke function) and the generating series counting ideals of a certain ring. Using this formula, we give an explicit estimate for such orbital integrals. We also derive an analogous formula for global fields, proving analytic properties of the Dedekind zeta function for orders in global fields.
\end{abstract}

\maketitle

\section{Introduction}
The goal of this note is two-fold. First, we give an explicit estimate of orbital integrals for the unit Hecke function on $\GL(n,F)$ where $F$ is a local non-archimedean field. Second, we study analytic properties of the Dedekind zeta function of an order in a global field. Though at first glance the two goals seem to be unrelated to each other, there is a key formula on which both of them rely. The key formula relates the cardinality of the ``compactified class group'' of some order $R$ in a local field to a local zeta function defined in terms of $R$.

\subsection{Dedekind zeta function for orders}\label{intro glob}
Let $E$ be a number field and $R\subset\calO_E$ be an order. Define the completed Dedekind zeta function of $R$ to be
\begin{equation}\label{define zeta}
\L_R(s):=D_{R}^{s/2}\Gamma_{E,\infty}(s)\sum_{M\subset R^\vee}(\#R^\vee/M)^{-s}.
\end{equation}
Here $R^\vee=\Hom_\ZZ(R,\ZZ)$ viewed as an $R$-module. The sum is over all nonzero $R$-submodules $M\subset R^\vee$. The number $D_{R}$ is the absolute discriminant of $R$ (see \S\ref{ss:disc}) and $\Gamma_{E,\infty}(s)$ is the usual Gamma-factor attached to the archimedean places of $E$ (see \cite{T}). When $R=\calO_E$, $R^\vee$ is an invertible $R$-module, and $\L_{\calO_E}(s)$ is the usual completed Dedekind zeta function of the number field $E$. One of the main results of this note says that $\L_R(s)$ shares many analytic properties with the usual Dedekind zeta function.

\begin{theorem}\label{th:global} Let $R\subset E$ be an order, then
\begin{enumerate}
\item The function $\L_R(s)$ admits a meromorphic continuation to all $s\in \CC$. The only poles of $\L_R(s)$ are simple poles at $s=0,1$ with residues
\begin{equation}\label{residue}
-\Res_{s=0}\L_R(s)=\Res_{s=1}\L_R(s)=\frac{2^{r_1}(2\pi)^{r_2}\#\Cl(R)\textup{Reg}_E}{\#\calO^\times_{E,\textup{tors}}}\sum_{\{M\}\in\Cl(R)\backslash \bCl(R)}\frac{\#(\calO^\times_E/\Aut(M))}{\#\Stab_{\Cl(R)}([M])}.
\end{equation}
Here, as usual, $r_1,r_2$ are the number of real and complex places of $E$, and $\textup{Reg}_E$ is the regulator of $E$. The finite set $\bCl(R)$ is the set of $E^\times$-homothety classes of fractional $R$-ideals, on which the class group $\Cl(R)$ acts. For other notations, see \S\ref{bCl}.


\item The function $\L_R(s)$ satisfies the functional equation
\begin{equation*}
\L_R(s)=\L_R(1-s).
\end{equation*}
\end{enumerate}
\end{theorem}
In fact, more general zeta functions (with $E$ replaced by a semisimple $\QQ$-algebra, $R$ an order in $E$ and $R^\vee$ replaced by an arbitrary finitely generated $R$-submodule of $E$) were introduced by Solomon \cite{So} and studied in depth in a series of papers by Bushnell and Reiner \cite{BR}, \cite{BR Res} and \cite{BR FE}. However, the fact that the functional equation holds for our specific zeta function $\L_R(s)$ does not seem to be found in the literature. The residue formula above is a special case of a result of Bushnell and Reiner \cite{BR Res}. 

Moreover, the meromorphic continuation and the same residue formula at $s=0$ is valid for more general zeta functions, with $R^\vee$ replaced by any fractional $R$-ideal in \eqref{define zeta}. 

We also have a function field analog of Theorem \ref{th:global},  see Theorem \ref{th:C}.

\subsection{Orbital integrals}\label{intro oi}
Another main result of this note concerns with estimates of orbital integrals. Let $F$ be a local non-archimedean field with ring of integers $\calO_F$ and residue field $k=\FF_q$. Let $G=\GL(n)$, viewed as a split reductive group over $\calO_F$. Let $\frg=\gl(n)$ be its Lie algebra.

Let $\gamma\in\frg(F)$ be a regular semisimple element. Let $T_\gamma$ be the centralizer of $\gamma$ in $G$, which is a maximal torus in $G$. Let $dg$ be the Haar measure on $G(F)$ with $\vol(dg,G(\calO_F))=1$. Let $T_c$ be the maximal compact subgroup of $T(F)$. Let $dt$ be the Haar measure on $T_\gamma$ with $\vol(dt,T_c)=1$. The two measures $dg$ and $dt$ together induce a measure $d\mu=\frac{dg}{dt}$ on the coset space $T_\gamma(F)\backslash G(F)$ that is right invariant under $G(F)$. The {\em orbital integral} of $\gamma$ is
\begin{equation*}
O_\gamma=\int_{T_\gamma(F)\backslash G(F)}\one_{G(\calO_F)}(g^{-1}\gamma g)d\mu(g).
\end{equation*}
where $\one_{G(\calO_F)}$ is the characteristic function of $G(\calO_F)\subset G(F)$.

\subsection{The polynomials $M_{\delta,r}$ and $N_{\delta,r}$} For each pair of integers $\delta\geq0,r\geq1$, we introduce a monic polynomial of degree $\delta$:
\begin{equation}\label{define M}
M_{\delta,r}(x)=\sum_{|\l|\leq\delta,m_1(\l)<r}x^{\delta-\ell(\l)}+\sum_{\delta-r\leq|\l|<\delta}x^{|\l|-\ell(\l)}.
\end{equation}
Here the summations are over partitions $\l$ (including the empty partition). For a partition $\l=\l_1\geq\l_2\geq\cdots$, we use $|\l|=\sum_i\l_i$ to denote its size, $\ell(\l)$ for the number of (nonzero) parts and $m_1(\l)$ for the number of times one appears as a part of $\l$. For example, the first summation in \eqref{define M} is over all partitions $\l$ with size $\leq\delta$ in which $1$ appears less than $r$ times.

We list the first few values of the polynomial $M_{\delta,r}$: 
\begin{eqnarray*}
&& M_{0,r}(x)=1; M_{1,1}(x)=x+1; M_{1,r}(x)=x+2 \textup{ for }r\geq2; \\
&& M_{2,1}(x)=x^2+x+1; M_{2,2}(x)=x^2+2x+2; M_{2,r}(x)=x^2+2x+3 \textup{ for }r\geq3;\\
&& M_{3,1}(x)=x^3+2x^2+x+1; M_{3,2}(x)=x^3+3x^2+2x+2; \\
&& M_{3,3}(x)=x^3+3x^2+3x+3; M_{3,r}(x)=x^3+3x^2+3x+4 \textup{ for }r\geq4; \cdots.
\end{eqnarray*}

We also define a polynomial $N_{\delta,r}\in\ZZ_{\geq0}[x]$, monic of degree $\delta$:
\begin{equation*}
N_{\delta,r}(x)=\begin{cases} x^{\delta}+x^{\delta-1}+\cdots+x^{\delta-r+1}+r & \textup{if }r\leq \delta; \\ x^{\delta}+x^{\delta-1}+\cdots+x+\delta+1 & \textup{if }r>\delta. \end{cases}
\end{equation*}

Another main result of this note is
\begin{theorem}\label{th:main} For a regular semisimple $\gamma\in\frg(F)$ whose characteristic polynomial $f(X)$ has coefficients in $\calO_F$, we have
\begin{equation*}
q^{\rho(\gamma)}\prod_{i\in B(\gamma)}N_{\delta_i,r_i}(q^{d_i})\leq O_\gamma\leq q^{\rho(\gamma)}\prod_{i\in B(\gamma)} M_{\delta_i,r_i}(q^{d_i}).
\end{equation*}
Here the products are over the set $B(\gamma)$ of irreducible factors $f_i(X)$ of $f(X)$, the integers $\delta_i,r_i,d_i$ and $\rho(\gamma)$ are certain invariants determined by the $f_i(X)$'s, see \S\ref{ss:inv}. Both the upper bound and the lower bound are monic polynomials of degree $\delta$ in $q$ with positive integer coefficients.
\end{theorem}

We spell out the special case when $\gamma$ is elliptic and that the field extension $F[\gamma]/F$ is totally ramified. In this case, we have
\begin{equation*}
q^\delta+1\leq O_\gamma \leq \sum_{|\l|\leq \delta, \textup{ with all parts}\geq2}q^{\delta-\ell(\l)}+\sum_{|\l|=\delta-1}q^{\delta-1-\ell(\l)}.
\end{equation*}

\subsection{Method of proof} The proofs of both theorems are based on a local statement, Theorem \ref{th:local}, whose proof is a simple exercise following Tate's thesis \cite{T}. To get the estimate in Theorem \ref{th:main}, one needs a formula for the number of ideals of a given colength in the two-dimensional ring $\calO_F[[X]]$, which is worked out in Proposition \ref{p:2dim}.

\section{Local study of zeta functions for orders}\label{s:local}

\subsection{The setup}\label{ss:R}
Let $F$ be a local non-archimedean field with ring of integers $\calO_F$ and residue field $k=\FF_q$. Let $|\cdot|:F^\times\to q^\ZZ$ be the standard absolute value:  $|\pi|=q^{-1}$ for any uniformizer $\pi$ in $F$.

Let $E$ be a finite-dimensional reduced $F$-algebra. We extend the absolute value $|\cdot|$ to $E^\times$ by pre-composing with the norm $N_{E/F}$. Write $E$ as a product of fields $E=\prod_{i\in B}E_i$ where $B$ is some finite index set. Let $\tilk_i$ be the residue field of $E_i$ and $n_i=[\tilk_i:k]$ for $i\in B$. The rings of integers of $E$ (resp. $E_i$) is $\calO_E$ (resp. $\calO_{E_i}$).

An {\em order} in $E$ is finitely generated $\calO_F$-algebra $R\subset E$ with $R\otimes_{\calO_F}F=E$. A {\em fractional $R$-ideal} is a finitely generated $R$-submodule $M\subset E$ such that $M\otimes_{\calO_F}F=E$. 

Let $X_R$ be the set of fractional $R$-ideals, on which $E^\times$ acts via multiplication. An $E^\times$-orbit in $X_R$ is called an {\em ideal class}, and we denote by $\bCl(R)=E^\times\backslash X_R$ the set of ideal classes of $R$. The usual ideal class group is the quotient $\Cl(R)=E^\times\backslash X^{\circ}_R$ where $X^\circ_R$ is the set of invertible fractional ideals. We may view $\bCl(R)$ as a ``compactification'' of the group $\Cl(R)$.

Choose a finite free $\ZZ$-module $\L\subset E^\times$ complementary to $\calO_E^\times$, for example $\L=\prod_i\pi_i^\ZZ$ for uniformizers $\pi_i\in E_i$. Then $\L$ acts freely on $X_R$ with finitely many orbits. The cardinality of the orbit set $\L\backslash X_R$ is independent of the choice of $\L$ complementary to $\calO_E$.

For any two fractional ideals $M_1$ and $M_2$, their relative ($\calO_F$-)length $[M_1:M_2]$ is defines as
\begin{equation}\label{rel leng}
[M_1:M_2]:=\leng_{\calO_F}(M_1/M_1\cap M_2)-\leng_{\calO_F}(M_2/M_1\cap M_2).
\end{equation}

The normalization of $R$ is $\calO_E$. The Serre invariant for $R$ is $\delta_R=[\calO_E:R]$, which we often abbreviate as $\delta$.

\subsection{Duality}\label{ss:dual} Let $\frd_{E/F}$ be the different ideal of $E/F$ (which is an ideal in $\calO_E$). Choose an $\calO_E$-generator of $c\in\frd_{E/F}$. On $E$ we define a modified trace paring $(\cdot,\cdot):E\times E\to F$ defined by
\begin{equation}\label{tr}
(x,y)=\Tr_{E/F}(c^{-1}xy).
\end{equation}
For any fractional $R$-ideal $M$, let $M^\vee=\{x\in E|(x,M)\subset\calO_F\}$, which is also a fractional ideal of $R$. 

The operation $M\mapsto M^\vee$ is an involution on the set of fractional $R$-ideals. The insertion of $c^{-1}$ in the definition of the modified trace pairing ensures that $\calO^\vee_E=\calO_E$. For any two fractional ideals $M_1,M_2\subset E$, we have
\begin{equation*}
[M_1:M_2]+[M^\vee_2:M^\vee_1]=0.
\end{equation*}
We have  $R\subset\calO_E\subset R^\vee$ with $[\calO_E:R]=[R^\vee:\calO_E]=\delta$.

\subsection{Local zeta function}\label{ss:J} For each $j\geq0$, let $\Quot^j_{R^\vee}$ be the set of fractional $R$-ideals $M\subset R^\vee$ with $[R^\vee:M]=j$. Let
\begin{equation}\label{define J}
J_R(s)=\sum_{j\geq0}\#\Quot^j_{R^\vee}\cdot q^{-js}.
\end{equation}
We define
\begin{eqnarray}\label{V}
V_R(s):=q^{\delta s}\prod_{i\in B}(1-q^{-n_is}),
\\
\label{tilJ}
\tilJ_R(s):=V_R(s)J_R(s).
\end{eqnarray}

\begin{remark}\label{r:h} Let $\Hilb^j_R$ be the set of ideals $I\subset R$ such that $[R:I]=j$. If $R$ is Gorenstein, i.e., $R^\vee$ is an invertible $R$-module, then $\Quot^j_{R^\vee}$ is naturally in bijection with $\Hilb^j_R$. Therefore, in the Gorenstein case, $J_R(s)$ is the generating series made by counting points on the ``Hilbert schemes of $R$''.
\end{remark}

\begin{theorem}\label{th:local} We have
\begin{enumerate} 
\item The function $\tilJ_R(s)$ is of the form
\begin{equation}\label{JP}
\tilJ_R(s)=q^{\delta s}P_R(q^{-s})
\end{equation}
for some polynomial $P_R(t)\in 1+t\ZZ[t]$ of degree $2\delta_R$.

\item $\tilJ_R(0)=\tilJ_R(1)=\#(\L\backslash X_R)$. Equivalently, $P_R(1)=q^\delta P_R(q^{-1})=\#(\L\backslash X_R)$.

\item There is a functional equation
\begin{equation}\label{FE}
\tilJ_R(s)=\tilJ_R(1-s).
\end{equation}
Or equivalently $P_R(t)=(qt^2)^{\delta}P_R(q^{-1}t^{-1})$. 
\end{enumerate}
\end{theorem}

\begin{remark} Let $M$ be a fractional  $R$-ideal. We may define the function $J_R(s)$ using $M$ instead of $R^\vee$, i.e., setting
\begin{equation*}
J^M_R(s):=\sum_{j\geq0}\#\Quot^j_{M}\cdot q^{-js},
\end{equation*}
and $\tilJ^M_R(s):=V_R(s)J^M_R(s)$. Then $\tilJ^M_R(s)$ is still of the form $q^{\delta s}P^M_R(q^{-s})$ for some polynomial $P^M_R(t)$ and
\begin{equation*}
\tilJ^M_R(0)=P^M_R(1)=\#(\L\backslash X_R)
\end{equation*}
still holds. However, in general, $P^M_R(t)$ will not be of degree $2\delta_R$, and it will not satisfy the functional equation as in Theorem \ref{th:local}(3).
\end{remark}

Before proving the theorem, we give some examples.

\begin{exam}\label{ex:n} Let $n\geq2$, $E=F^n$ and $R\subset\calO_F^n$ consist of $(x_1,\cdots,x_n)$ such that $x_1\equiv x_2\equiv\cdots\equiv x_n\mod\pi$. Then $\delta=n-1$. The dual $R^\vee$ consists of triples $(y_1,\cdots,y_n)\in(\pi^{-1}\calO_F)^n$ such that $y_1+\cdots+y_n\in\calO_F$. Note that $R$ is not Gorenstein when $n\geq3$. We have a formula for $P_R(t)$ by direct calculation:
\begin{equation}\label{P_R}
P_R(t)=\sum_{r=0}^n\binom{n}{r}(1-t)^rt^{n-r}\sum_{c=0}^{n-1}t^c\left(\mbox{coeff. of $t^{n-c-1}$ in }\frac{(1-t)^{n-r}}{\prod_{j=0}^{c+\ep}(1-q^jt)}\right).
\end{equation}
Here $\ep=0$ if $r>0$ and $\ep=1$ if $r=0$. It is an interesting combinatorial problem to show directly that this polynomial satisfies the functional equation $q^{n-1}t^{2n-2}P_R(q^{-1}t^{-1})=P_R(t)$.

For a geometric description of $\L\backslash X_R$ when $F$ is a function field, see Remark \ref{r:XR}.
\end{exam}

\begin{exam} Continuing the above example and take $n=3$, a direct calculation shows
\begin{equation*}
P_R(t)=1+(q-2)t+t^2+(q^2-2q)t^3+q^2t^4.
\end{equation*}
We try with other definitions of the zeta function by replacing $R^\vee$ in the definition by other $R$-modules. For example, if we define
\begin{equation*}
J^\sharp_R(s):=\sum_{j\geq0}\#\Quot^j_{\calO^{3}_F}\cdot q^{-js}.
\end{equation*}
Then the corresponding polynomial $P^\sharp_R(t)$ is
\begin{equation*}
P^\sharp_R(t)=1+(q^2+q-2)t+(q^2-2q+1)t^2.
\end{equation*}
which does not satisfy the functional equation. If we define
\begin{equation*}
J^\flat_R(s):=\sum_{j\geq0}\#\Hilb^j_{R}\cdot q^{-js}.
\end{equation*}
Then the corresponding polynomial $P^\flat_R(t)$ is
\begin{equation*}
P^\flat_R(t)=1-2t+(q^2+q+1)t^2+(q^2-2q)t^3.
\end{equation*}
which again does not satisfy the functional equation. However, all these polynomials $P_R,P^\sharp_R$ and $P^\flat_R$ have the same special value $2q^2-q$ at $t=1$, which is the number of fractional $R$-ideals up to multiplication by the three uniformizers. 
\end{exam}

The rest of the section is devoted to the proof of Theorem \ref{th:local}. We first need some preparation. 

\subsection{Contribution of an ideal class} For a fractional $R$-ideal $M$, the ideal class containing it is the set $[M]=\{xM, x\in E^\times\}$. Let $c^j([M])$ be the cardinality of $\Quot^j_{R^\vee}\cap[M]$. The contribution of the ideal class $[M]$ to $J_R(s)$ is
\begin{equation*}
J_R([M],s):=\sum_jc^j([M])q^{-js}.
\end{equation*}
The contribution of $[M]$ to $\tilJ_R(s)$ is $\tilJ_R([M],s):=V_R(s)J_R([M],s)$.

Fix a Haar measure $dx$ on the additive group $E$ such that $\vol(dx,\calO_E)=1$. So for any fractional ideal $M$ we have $\vol(dx,M)=q^{[M:\calO_E]}$. We also get a Haar measure $d^\times x=dx/|x|$ on the multiplicative group $E^\times$.

For a fractional $R$-ideal $M$, let $\Aut(M)=\{x\in E^\times|xM=M\}$. Then $\Aut(M)$ is a compact open subgroup of $E^\times$ and hence has a nonzero finite volume under the Haar measure $d^\times x$.

\begin{lemma}\label{l:CM} We have
\begin{equation}\label{Cint}
J_R([M],s):=q^{-[R^\vee:M]s}\vol(d^\times x,\Aut(M))^{-1}\int_{E^\times}\one_{M^\vee}(x)|x|^sd^\times x.
\end{equation}
Note the right side is independent of the choice of the representative $M$ in the class $[M]$ and the choice of the measure $d^\times x$.
\end{lemma}
\begin{proof} Let $M^\vee_j\subset M^\vee$ consist of elements $x\in M^\vee$ with $|x|=q^{-j}$. Then $M^\vee_j$ admits a free action of $\Aut(M)$ via multiplication. The right side of \eqref{Cint} breaks up into a sum
\begin{equation*}
q^{-[R^\vee:M]s}\sum_{j\in\ZZ}\#(M^\vee_j/\Aut(M))q^{-js}.
\end{equation*}
The map $M^\vee_j\ni x\mapsto xM\in\Quot^{j+[R^\vee:M]}_{R^\vee}$ gives a bijection between $M^\vee_j/\Aut(M)$ and $\Quot^{j+[R^\vee:M]}_{R^\vee}\cap[M]$. Therefore the right side of \eqref{Cint} can be further written as
\begin{equation*}
q^{-[R^\vee:M]s}\sum_{j\in\ZZ}\#(M^\vee_j/\Aut(M))q^{-js}=\sum_{j\in\ZZ}c_{j+[R^\vee:M]}q^{-(j+[R^\vee:M])s}=J_R([M],s).
\end{equation*}
\end{proof}

Summing up the contributions from all ideal classes, we get
\begin{equation}\label{qJ}
J_R(s):=\sum_{[M]\in \bCl(R)}q^{-[R^\vee:M]s}\vol(d^\times x,\Aut(M))^{-1}\int_{E^\times}\one_{M^\vee}(x)|x|^sd^\times x.
\end{equation}

\subsection{Proof of Theorem \ref{th:local}(1)} In view of Lemma \ref{l:JHilb}, we set 
\begin{equation}\label{t}
P_R(t):=\prod_{i\in B}(1-t^{n_i})\sum_{j\geq0}\#\Quot^j_{R^\vee}\cdot t^j.
\end{equation}
This is clearly in $1+t\ZZ[[t]]$. To show this is a polynomial, it suffices to show that $\tilJ_R([M],s)\in\QQ[q^s,q^{-s}]$ for each ideal class $[M]$.

Using \eqref{qJ}, we have
\begin{equation*}
\tilJ_R([M],s)=\frac{q^{(\delta-[R^\vee:M])s}}{\vol(d^\times x,\Aut(M))}\int_{E^\times}\one_{M^\vee}(x)|x|^s\prod_{i\in B}(1-q^{-n_is})d^\times x
\end{equation*}
Note that
\begin{equation}\label{OE}
\vol(d^\times x,\calO^\times_E)\prod_{i\in B}(1-q^{-n_is})^{-1}=\int_{E^\times}\one_{\calO_E}(x)|x|^s
d^\times x.
\end{equation}
Therefore
\begin{eqnarray}\label{diff}
\notag\tilJ_R([M],s)&=&\frac{q^{(\delta-[R^\vee:M])s}}{\vol(d^\times x,\Aut(M))}\cdot\\
&&\left(\vol(d^\times x,\calO^\times_E)
+\int_{E^\times}(\one_{M^\vee}(x)-\one_{\calO_E}(x))|x|^s\prod_{i\in B}(1-q^{-n_is})d^\times x\right).
\end{eqnarray}
The integral above now only involves finitely many possible valuations of $x$, hence belongs to $\QQ[q^{s},q^{-s}]$, and therefore $\tilJ_R([M],s)\in\QQ[q^{s},q^{-s}]$. This shows that $P(t)$ is a polynomial. The fact that $\deg P_R(t)=2\delta$ will follow once we prove the functional equation below.

\subsection{Proof of Theorem \ref{th:local}(3)} We will in fact prove the following refinement of the functional equation, which implies \eqref{FE} by summing up the contributions from all ideal classes.

\begin{lemma}
For each ideal class $[M]\in \bCl(R)$, we have
\begin{equation*}
\tilJ_R([M],s)=\tilJ_R([M^\vee],1-s).
\end{equation*}
\end{lemma}
\begin{proof} Using notation from Tate's thesis \cite[\S2.4]{T}, we write the local zeta function
\begin{equation*}
\zeta(f,||^s):=\int_{E^\times}f(x)|x|^sd^\times x.
\end{equation*}
for compactly supported locally constant functions $f:E\to \CC$. Using Lemma \ref{l:CM} and \eqref{OE}, we can then write
\begin{equation}\label{tilC}
\tilJ_R([M],s)=q^{(\delta-[R^\vee:M])s}\frac{\vol(d^\times x,\calO^\times_E)}{\vol(d^\times x,\Aut(M))}\frac{\zeta(\one_{M^\vee},||^{s})}{\zeta(\one_{\calO_E},||^{s})}.
\end{equation}
Fix an additive character $\psi: F\to \CC^\times$ which is trivial on $\calO_F$ but nontrivial on $\pi^{-1}\calO_F$, we can define the Fourier transform $f\mapsto \wh{f}$ for functions $f$ on $E$:
\begin{equation*}
\wh{f}(y)=\int_Ef(x)\psi((x,y))dx.
\end{equation*}
Here $(x,y)$ is the modified trace pairing in \eqref{tr}.
By \cite[Theorem 2.4.1]{T} (local functional equation), we have
\begin{equation}\label{Tate}
\frac{\zeta(\one_{M^\vee},||^{s})}{\zeta(\one_{\calO_E},||^{s})}=\frac{\zeta(\wh{\one}_{M^\vee},||^{1-s})}{\zeta(\wh{\one}_{\calO_E},||^{1-s})}.
\end{equation}
A straightforward calculation shows that for any fractional $R$-ideal $M\subset E$, we have
\begin{equation*}
\wh{\one}_{M^\vee}=q^{[\calO_E:M]}\one_M.
\end{equation*}
Since $\calO_E$ is self-dual,  we have $\wh{\one}_{\calO_E}=\one_{\calO_{E}}$. Therefore \eqref{Tate} and \eqref{tilC} together imply
\begin{eqnarray}\label{local ratio}
\tilJ_R([M],s)&=&q^{(\delta-[R^\vee:M])s}\frac{\vol(d^\times x,\calO^\times_E)}{\vol(d^\times x,\Aut(M))}\frac{\zeta(\one_{M^\vee},||^{s})}{\zeta(\one_{\calO_E},||^{s})}\\
\notag&=&q^{(\delta-[R^\vee:M])s}\frac{\vol(d^\times x,\calO^\times_E)}{\vol(d^\times x,\Aut(M^\vee))}\frac{q^{[\calO_E:M]}\zeta(\one_{M},||^{1-s})}{\zeta(\one_{\calO_E},||^{1-s})}
\end{eqnarray}
The exponent of $q$ on the right side is
\begin{eqnarray*}
(\delta-[R^\vee:M])s+[\calO_E:M]&=&(\delta-[R^\vee:M]+[\calO_E:M])s+[\calO_E:M](1-s)\\
&=&[\calO_E:M](1-s)=(\delta-[R^\vee:M^\vee])(1-s).
\end{eqnarray*}
Also $\Aut(M)=\Aut(M^\vee)$, therefore the right side of \eqref{local ratio} is exactly $\tilJ_R([M^\vee],1-s)$. 
\end{proof}

\subsection{Proof of Theorem \ref{th:local}(2)} Plugging in $s=0$ in \eqref{diff} (which is now legitimate because it only involves finitely many powers of $q^{-s}$), and then sum over all ideal classes $[M]$, we get
\begin{equation}\label{jr}
\tilJ_R(0)=\sum_{[M]\in \bCl(R)}\frac{\vol(d^\times x,\calO^\times_E)}{\vol(d^\times x,\Aut(M))}=\sum_{[M]\in \bCl(R)}\#(\calO^\times_E/\Aut(M)).
\end{equation}
Now we claim that the right side is $\#(\L\backslash X_R)$. In fact, consider the quotient map $p:\L\backslash X_R\to E^\times\backslash X_R=\bCl(R)$. We have
\begin{equation*}
\#p^{-1}([M])=\#(E^\times/\L\Aut(M))=\#(\calO^\times_E/\Aut(M)).
\end{equation*}
Therefore \eqref{jr} implies $\tilJ_R(0)=\#(\L\backslash X_R)$. Using the functional equation \eqref{FE}, we get also $\tilJ_R(1)=\#(\L\backslash X_R)$.

This finishes the proof of Theorem \ref{th:local}.

\section{Dedekind zeta function for orders}
We would like to study the global field analog of the function $\tilJ_R(s)$, i.e., the Dedekind zeta function of orders. We will treat number fields and function fields separately, and prove an analog of Theorem \ref{th:local} for in both situations. 

\subsection{The number field setting}
Let $E$ now be a number field and $R\subset E$ be an order (i.e., a finitely generated $\ZZ$-algebra whose fraction field is $E$). Let $|R|$ be the set of maximal ideals of $R$. For each $v\in |R|$, let $R_v$ the $v$-adic completion of $R$ with residue field $k_v$ and ring of total fractions $E_v$. Let $q_v=\#k_v$. 

We recall the completed Dedekind zeta function $\L_R(s)$ defined in \eqref{define zeta}:
\begin{equation*}
\L_R(s)=D_R^{s/2}\left(\pi^{-s/2}\Gamma(s/2)\right)^{r_1}\left((2\pi)^{1-s}\Gamma(s)\right)^{r_2}\sum_{M\subset R^\vee}(\#R^\vee/M)^{-s}
\end{equation*}
where $r_1$ (resp. $r_2$) are the number of real (resp. complex) places of $E$, and the sum is over all nonzero $R$-submodules $M\subset R^\vee$. The absolute discriminant $D_R$ is defined below.

\subsection{Discriminant}\label{ss:disc} The trace pairing $(x,y)\mapsto\Tr_{E/\QQ}(xy)$ allows to identify $E$ and the $\QQ$-linear dual of $E$. In particular, the dual $R^\vee=\Hom_{\ZZ}(R,\ZZ)$ can also be viewed as a fractional ideal in $E$ which contains $R$. We define the absolute discriminant $D_{R}$ to be $\#(R^\vee/R)$. We have
\begin{equation}\label{disc}
D_{R}=\#(\calO_E/R)^2|\Delta_{E/\QQ}|.
\end{equation}

\subsection{Global ideal classes}\label{bCl} Let $\AA^\times_{E,f}$ be the finite part of the id\`ele group of $E$. The class group $\Cl(R)$ of $R$ is the group \begin{equation*}
\Cl(R)=E^\times\backslash\AA^\times_{E,f}/\prod_{v\in|R|}R^\times_v.
\end{equation*}
For $v\in |R|$, let $X_v$ be the set of fractional $R_v$-ideals. The orbit space
\begin{equation*}
\bCl(R):=E^\times\backslash\prod_{v\in |R|}X_v
\end{equation*}
classifies fractional $R$-ideals up to multiplication by $E^\times$, or, equivalently, isomorphism classes of torsion-free $R$-modules of rank one. For each fractional ideal $M\subset E$, we define $\Aut(M):=\{x\in E^\times|xM=M\}$, which only depends on the ideal class $[M]$.

The class group $\Cl(R)$ acts on $\bCl(R)$ and $\bCl(R)$ can be viewed as a ``compactification'' of $\Cl(R)$. The orbit space $\Cl(R)\backslash\bCl(R)$ is the set of {\em genera} of $R$, i.e., the equivalence classes of fractional $R$-ideals under local isomorphism. For a fractional $R$-ideal $M$, we denote the $\Cl(R)$-orbit containing $[M]$ by $\{M\}\subset\bCl(R)$. Let $\Stab_{\Cl(R)}([M])$ be the stabilizer of $[M]$ under $\Cl(R)$, which only depends on the $\Cl(R)$-orbit of $[M]$.  Now we have explained all notations in the statement of Theorem \ref{th:global}.

\subsection{Proof of Theorem \ref{th:global}} For each $v\in |R|$, we apply the discussion in \S\ref{s:local} to the order $R_v\subset E_v$. We can then define the functions $J_v(s), \tilJ_v(s)$ and the polynomial $P_v(t)$. We also have the set $X_v$ of fractional $R_v$-ideals and choose a free abelian group $\L_v\subset E^\times_v$ complementary to $\calO^\times_{E_v}$, as in \ref{ss:R}. 
 
Both $\L_R(s)$ and $\L_E(s)$ (the usual completed Dedekind zeta function for the number field $E$) admit Euler products over $v\in|R|$ (which is a partition of finite places of $E$). We can express the ratio $\L_R(s)/\L_E(s)$ as a product of  ratios of  local zeta functions
\begin{equation}\label{ratio}
\frac{\L_R(s)}{\L_E(s)}=(\#\calO_E/R)^s\prod_{v\in|R|}J_v(q_v^{-s})\prod_{u\in|\calO_E|, u\mapsto v}(1-q_u^{-s})=\prod_{v\in |R|}q_v^{\delta_vs}P_v(q_v^{-s})=\prod_{v\in |R|}\tilJ_v(s).
\end{equation}
Here we used \eqref{disc} to compute the ratio of the discriminant terms. By Theorem \ref{th:local}(1), this ratio is an entire function in $s$. Since $\L_E(s)$ has a meromorphic continuation to $\CC$ with simple poles at $s=0,1$,  $\L_R(s)$ also has a meromorphic continuation to $\CC$ with at most simple poles at $s=0,1$. 

The functional equation in Theorem \ref{th:global}(2) follows from \eqref{ratio} and the local functional equation \eqref{FE}.

Let us compute the residue of $\L_R(s)$ at $s=1$. The residue at $s=0$ follows from the functional equation. Using \eqref{ratio} and Theorem \ref{th:local}(2), we have
\begin{equation}\label{pre res}
\Res_{s=1}\L_R(s)=\Res_{s=1}\L_E(s)\cdot\prod_{v\in|R|}\#(\L_v\backslash X_v).
\end{equation}
The residue formula for $\L_E(s)$ in \cite[Main Theorem 4.4.1]{T} says (note the factor $|\Delta_{E/\QQ}|^{1/2}$ disappears because of our definition of $\L_E(s)$) 
\begin{equation*}
\Res_{s=1}\L_E(s)=\frac{2^{r_1}(2\pi)^{r_2}\#\Cl(\calO_E)\textup{Reg}_E}{\#\calO^\times_{E,\textup{tors}}}
\end{equation*}
Multiplying the above two formulae, we see that in order to prove \eqref{residue}, we only need to show that
\begin{equation}\label{df}
\#\Cl(\calO_E)\prod_{v\in|R|}\#(\L_v\backslash X_v)=\#\Cl(R)\sum_{\{M\}\in\Cl(R)\backslash\bCl(R)}\frac{\#(\calO^\times_E/\Aut(M))}{\#\Stab_{\Cl(R)}([M])}.
\end{equation}
We first break the left side into a sum over the genera. Note that $\Cl(R)\backslash\bCl(R)=\prod(E^\times_v\backslash X_v)$ and we have a natural map $c:\prod(\L_v\backslash X_v)\to\prod(E^\times_v\backslash X_v)=\Cl(R)\backslash\bCl(R)$. For each genus $\{M\}$, we may represent it by $(M_v)_{v\in |R|}$ with $M_v\in X_v$. The preimage $c^{-1}(\{M\})$ is identified with $\prod(E^\times_v/\L_v\Stab(M_v))=\prod(\calO^\times_{E_v}/\Stab(M_v))$, where $\Stab(M_v)$ is the stabilizer of $M_v\in X_v$ under the action of $E^\times_v$. Hence we have
\begin{equation*}
\prod_{v\in|R|}\#(\L_v\backslash X_v)=\sum_{\{M\}\in\Cl(R)\backslash\bCl(R)}\#\prod(\calO^\times_{E_v}/\Stab(M_v)).
\end{equation*}
On the other hand, the term $\#\Cl(R)\#\Stab_{\Cl(R)}([M])^{-1}$ on the right side of \eqref{df} is the cardinality of $\{M\}=E^\times\backslash\prod(E^\times_v/\Stab(M_v))$.
Therefore, to prove \eqref{df}, we only need to show for each genus $\{M\}$ that
\begin{equation}\label{df1}
\frac{\#\prod(\calO^\times_{E_v}/\Stab_v)}{\#(\calO^\times_E/\Aut(M))}=\frac{\#(E^\times\backslash\prod(E^\times_v/\Stab_v))}{\#\Cl(\calO_E)}.
\end{equation}
Here we have abbreviated $\Stab(M_v)$ to $\Stab_v$. Indeed, \eqref{df1} follows from the exact sequence
\begin{equation*}
1\to\calO^\times_E/\Aut(M)\to\prod(\calO^\times_{E_v}/\Stab_v)\to E^\times\backslash\prod(E^\times_v/\Stab_v)\to\Cl(\calO_E)\to1.
\end{equation*}
where the last homomorphism is the natural projection $E^\times\backslash\prod(E^\times_v/\Stab_v)\to E^\times\backslash\prod(E^\times_v/\calO^\times_{E_v})=\Cl(\calO_E)$. This finishes the proof of Theorem \ref{th:global}.

\subsection{The function field setting} In the function field case, we will use geometric language. Let $C$ be a projective curve over a finite field $k=\FF_q$ which is geometrically integral (i.e., $C\otimes_k\kbar$ is irreducible and reduced). Let $E=k(C)$ be the function field of $C$. Let $|C|$ be the set of closed points of $C$. For each $v\in |C|$, let $R_v$ be the completed local ring at $v$ with ring of fractions $E_v$ and residue field $k_v=\FF_{q_v}$. We again use $X_v$ to denote the set of fractional $R_v$-ideals.

The arithmetic genus of $C$ is defined as $g_a=g_a(C)=\dim_k\cohog{1}{C,\calO_C}$. Let $\nu:\tilC\to C$ be the normalization, then the Serre invariant of $C$ is $\delta_C=\dim_k(\calO_{\tilC}/\calO_C)$. The quantity $q^{g_a-1}$ is the analog of $D^{1/2}_{R}$ for an order $R$ in a number field, and the relation $g_a(C)=g(\tilC)+\delta_C$ is the analog of \eqref{disc}.

Let $\omega=\omega_C$ be the dualizing sheaf of $C$. This is a coherent sheaf which is generically a line bundle. Let $\Quot^j_\omega$ be the Quot-scheme of length $j$ quotients of $\omega$. Define the complete Dedekind zeta function for $C$ to be
\begin{equation*}
\L_C(s)=q^{(g_a-1)s}\sum_{j\geq0}\#\Quot^j_\omega(k)q^{-js}.
\end{equation*}
When $C$ is a smooth projective curve, we have
\begin{equation*}
\L_C(s)=q^{(g_a-1)s}Z(C/\FF_q,s):=q^{(g_a-1)s}\exp\left(\sum_{i\geq1}\#C(\FF_{q^i})\frac{q^{-is}}{i}\right).
\end{equation*}
which is essentially the usual zeta function of $C$ introduced by E.Artin. Hence we view $\L_C(s)$ as a generalization of the zeta function for a possibly singular curve.

\subsection{The compactified Jacobian} In \cite{AK}, Altman and Kleiman defined the compactified Picard scheme $\cPic_C$ of $C$ classifying torsion-free coherent sheaves on $C$ of generic rank one. The compactified Jacobian $\cJac_C$ is the component of $C$ consisting of coherent sheaves with the same degree as the structure sheaf $\calO_C$. We have a natural bijection
\begin{equation}\label{adele}
\cJac_C(k)=E^\times\backslash(\prod_{v\in |C|}X_v)^1.
\end{equation}
Here, $(-)^1$ again means the norm one part. The norm is given by the product of local norms $X_v\to q_v^\ZZ$ sending a fractional $R_v$-ideal $M$ to $q_v^{[M:R_v]}$. To see the bijection, consider a coherent sheaf $\calF\in\cJac_C(k)$ together with a trivialization of it at the generic point $\iota:\calF|_{\Spec E}\cong E$. Then this datum determines a fractional $R_v$-ideal $\calF\otimes R_v\subset\calF\otimes E_v\xrightarrow{\iota_v} E_v$, hence a point $(\calF\otimes R_v)_{v\in|C|}\in\prod X_v$. Since $\calF$ has the same degree as $\calO_C$, we have $(\calF\otimes R_v)\in(\prod_vX_v)^1$. Changing the trivialization $\iota$ amounts to translating $(\calF\otimes R_v)$ by an element in $E^\times$, hence the bijection \eqref{adele}.

The following theorem is a global analog of Theorem \ref{th:local} for function fields.
\begin{theorem}\label{th:C} We have
\begin{enumerate}
\item The function $\L_C(s)$ is of the form
\begin{equation*}
\L_C(s)=q^{(g_a-1)s}\frac{P_C(q^{-s})}{(1-q^{-s})(1-q^{1-s})}.
\end{equation*}
where $P_C(t)\in1+t\ZZ[t]$ is a polynomial of degree $2g_a$ with special values
\begin{equation*}
P_C(1)=q^{g_a}P_C(q^{-1})=\#\cJac_C(k).
\end{equation*}
\item The function $\L_C(s)$ satisfies the functional equation
\begin{equation*}
\L_C(s)=\L_C(1-s).
\end{equation*}
Equivalently, $P_C(t)=(qt^2)^{g_a}P_C(q^{-1}t^{-1})$.
\end{enumerate}
\end{theorem}
The proof is similar to that of Theorem \ref{th:global} using the local results in Theorem \ref{th:local}. We omit the proof here. In \cite{MY}, in the case $C$ has only planar singularities, we gave a cohomological interpretation of the coefficients of $P_C(t)$, which does not seem to have a number field analog.

\section{Orbital integrals and local zeta functions}\label{s:oi} 

Throughout this section, $F$ is a local non-archimedean field. We continue to use the notation set up in \S\ref{intro oi}.

\subsection{Invariants attached to $\gamma$}\label{ss:inv} Let $\gamma\in\frg(F)$ be a regular semisimple element as in \S\ref{intro oi}. Let $f(X)\in F[X]$ be the characteristic polynomial of $\gamma$. Note that $O_\gamma\neq0$ only if $f(X)\in\calO_F[X]$.

Let $R=\calO_F[X]/(f(X))$. This is an $\calO_F$-algebra with ring of fractions $E=F[X]/(f(X))$. The pair $(R\subset E)$ is considered in \S\ref{ss:R}, and we shall use the notation from there. In particular, $E$ is a product of fields $\prod_{i\in B(\gamma)} E_i$ where $B(\gamma)$ is an index set in bijection with the irreducible  factors $f_i(X)$ of $f(X)$. We have the Serre invariant $\delta=\delta_R$.

Set $R_i=\calO_F[X]/(f_i(X))$ then $E_i=F[X]/(f_i(X))$ is the field of fractions of $R_i$. Let $\tilk_i$ and $k_i$ be the residue fields of $\calO_{E_i}$ and $R_i$ respectively. Summarizing
\begin{equation*}
\xymatrix{\calO_F\ar[r]\ar[d] & R_i\ar[r]\ar[d] & \calO_{E_i}\ar[d]\\
k\ar[r]^{d_i} & k_i\ar[r]^{r_i} & \tilk_i}
\end{equation*}
where the integers $d_i$ and $r_i$ stand for the degrees of the respective field extensions, and we let $n_i=d_ir_i=[\tilk_i:k]$.

The normalization $\calO_E=\prod_i\calO_{E_i}$ of $R$ is also the maximal compact subgroup $T_c$ of $T_\gamma(F)$. We also define the Serre invariant of $R_i$ to be $\delta_i=\leng_{R_i}(\calO_{E_i}/R_i)$. Let
\begin{equation*}
\rho(\gamma):=\delta-\sum_{i\in B(\gamma)}d_i\delta_i
\end{equation*}
Then we have
\begin{equation*}
\rho(\gamma)=\sum_{\{i,j\}\subset B(\gamma),i\neq j}\val_F(\Res(f_i,f_j)).
\end{equation*}
where $\Res(\cdot,\cdot)$ is the resultant of two polynomials.

\subsection{Orbital integrals and lattices} Let $X_\gamma=\{g\in G(F)/G(\calO_F)|g^{-1}\gamma g\in\frg(\calO_F)\}$. Then $T_\gamma(F)$ acts on $X_\gamma$ by left translation. Choosing a free abelian group $\L_\gamma\in E^\times=T_\gamma(F)$ complementary to the maximal compact $T_c$. For example, we may fix a uniformizer $\pi_i$ for each $E_i$, and let $\L_\gamma=\prod_i\pi_i^\ZZ$. Then $\L_\gamma$ acts freely on $X_\gamma$. We have
\begin{equation}\label{OX}
O_\gamma=\#(\L_\gamma\backslash X_\gamma). 
\end{equation}

A {\em lattice} in $V=F^n$ is an $\calO_F$-submodule $M$ of $V$ of rank $n$. For example $L_0=\calO^n_F\subset F^n$ is a lattice. For two $\calO_F$-lattices $L_1,L_2\subset V$, we define their relative ($\calO_F$-)length using formula \eqref{rel leng}.

The map $G(F)/G(\calO_F)\ni g\mapsto g\calO^n_F$ gives a bijection between $G(F)/G(\calO_F)$ and the set of lattices in $V$. The subset $X_\gamma\subset G(F)/G(\calO_F)$ corresponds bijectively to lattices $L$ which are stable under $\gamma$ (i.e., $\gamma L\subset L$).

\subsection{Another orbital integral} Let $V^*$ be the dual vector space of $V$. It also contains a standard lattice $L^*_0$ spanned by the standard dual basis over $\calO_F$. Let $e^*_n$ be standard basis element $(0,\cdots,0,1)\in V^*$. The algebra $E$ acts on $V^*$ on the right (using the adjoint of $\gamma$), realizing $V^*$ as a one-dimensional vector space over $E$ with basis $e_n^*$. The $R$-translation of $e^*_n$ inside of $V^*$ gives a lattice $e^*_nR$. 

Inspired by Jacquet \cite{J}, we consider the following orbital integral with a complex parameter $s$:
\begin{equation*}
J_\gamma(s):=q^{[e^*_nR:L^*_0]s}\int_{G(F)}\one_{\frg(\calO_F)}(g^{-1}\gamma g)\one_{L^*_0}(e^*_ng)|\det(g)|^sdg.
\end{equation*}

\subsection{Relation with local zeta function} We may relate $J_\gamma(s)$ to the local zeta function $J_R(s)$ defined in \S\ref{ss:J}. Using the vector $e^*_n$ we may identify $V^*$ with $E$: $E\ni x\mapsto e^*_nx\in V^*$. Transporting the modified trace pairing \eqref{tr} to $V^*$ using this identification, we can also identify $V$ with $E$, such that under these identifications, the usual pairing between $V^*$ and $V$ becomes the modified trace pairing on $E$ itself.

\begin{lemma}\label{l:JHilb} We have
\begin{equation}\label{JHilb}
J_\gamma(s)=J_R(s),
\end{equation}
where $J_R(s)$ is defined in \eqref{define J}.
\end{lemma}
\begin{proof}
For each $g\in G(F)/G(\calO_F)$, the condition that $g^{-1}\gamma g\in\frg(\calO_F)$ and $e^*_ng\in L^*_0$ is the same as saying that the $\calO_F$-lattice $L=gL_0$ is stable under $\gamma$ and that $\jiao{e^*_nR,gL_0}\subset\calO_F$ (here $\jiao{\cdot,\cdot}$ means the pairing between $V^*$ and $V$). Under the identification of $V\cong E\cong V^*$ given above,  this is the same as saying that $gL_0$ is a fractional $R$-ideal contained in $R^\vee$. Moreover $[R^\vee:gL_0]=[R^\vee:L_0]+[L_0:gL_0]=-[e^*_nR:L^*_0]+\val_F(\det(g))$. Therefore\begin{eqnarray*}
J_\gamma(s)=q^{[e^*_nR:L^*_0]s}\sum_{j\in\ZZ}\#\Quot^j_{R^\vee}\cdot q^{(-j-[e^*_nR:L^*_0])s}=\sum_{j\geq0}\#\Quot^j_{R^\vee}\cdot q^{-js}.
\end{eqnarray*}
\end{proof}

Now we define $\tilJ_\gamma(s)$ using the same formula \eqref{tilJ}, using the invariants $\delta=\delta_R$ and $n_i=[\tilk_i:k]$.

Theorem \ref{th:local} then implies
\begin{cor}[of Theorem \ref{th:local}]\label{c:var} We have
\begin{enumerate} 
\item The function $\tilJ_\gamma(s)$ is of the form
\begin{equation*}
\tilJ_\gamma(s)=q^{\delta s}P_\gamma(q^{-s})
\end{equation*}
for some polynomial $P_\gamma(t)\in 1+t\ZZ[t]$ of degree $2\delta$.
 
\item $\tilJ_\gamma(0)=\tilJ_\gamma(1)=O_\gamma$. 

\item There is a functional equation
\begin{equation*}
\tilJ_\gamma(s)=\tilJ_\gamma(1-s).
\end{equation*}
Or equivalently $P_\gamma(t)=(qt^2)^{\delta}P_\gamma(q^{-1}t^{-1})$.
\end{enumerate}
\end{cor}

\begin{cor}\label{c:hilb} Suppose $\gamma$ is elliptic (i.e., $f(X)$ is irreducible) and the residue field of $R$ is $k$, then
\begin{equation}\label{hilb}
O_\gamma=\sum_{j=0}^{\delta-r-1}(q^{\delta-j}-q^{\delta-r-j})\#\Hilb^j_R+q^{r}\#\Hilb^{\delta-r}_R+\sum_{j=\delta-r+1}^{\delta-1}(q^{\delta-j}+1)\#\Hilb^j_R+\#\Hilb^{\delta}_R.
\end{equation}
\end{cor}
\begin{proof} As noticed  in Remark \ref{r:h}, in the situation $R$ is Gorenstein, which it is now, one can replace $\Quot^j_{R^\vee}$ by $\Hilb^j_R$. Direct calculation using the definition of $P_\gamma(t)$ gives
\begin{equation*}
P_\gamma(t)=\sum_{j=0}^{2\delta}(\#\Hilb^j_R-\#\Hilb^{j-r}_R)t^j.
\end{equation*}
The functional equation $P_\gamma(t)=(qt^2)^{\delta}P_\gamma(q^{-1}t^{-1})$ in Corollary \ref{c:var}(3) implies that for $j=0,\cdots,\delta$, we have
\begin{equation*}
q^{\delta-j}(\#\Hilb^j_R-\#\Hilb^{j-r}_R)=\#\Hilb^{2\delta-j}_R-\#\Hilb^{2\delta-j-r}_R.
\end{equation*}
Therefore,
\begin{equation}\label{Lef}
P_\gamma(1)=\#\Hilb^{\delta}_R-\#\Hilb^{\delta-r}_R+\sum_{j=0}^{\delta-1}(\#\Hilb^j_R-\#\Hilb^{j-r}_R)(1+q^{\delta-j}).
\end{equation}
Switching the order of summation,  the right side above transforms to the right side of \eqref{hilb}. By Corollary \ref{c:var}(2), $O_\gamma=P_\gamma(1)$, therefore \eqref{hilb} holds. 
\end{proof}

The rest of the note is devoted to the proof of Theorem \ref{th:main}.

\subsection{Reduction of Levi subgroups} Let $L=\prod_{i\in B}L_i\subset\GL(n)=G$ be a Levi subgroup (so $L_i\cong\GL(n_i)$ with $\sum_{i\in B}n_i=n$). Let $\gamma=(\gamma_i)\in L(F)$ with $\gamma_i\in L_i$. One can similarly define $X^L_\gamma=\prod_iX^{L_i}_{\gamma_i}$, on which $T_\gamma(F)=\prod_iT_{\gamma_i}(F)$ acts. We may fix a choice of $\L_{i}\subset T_{\gamma_i}(F)$ (complementary to the maximal compact in $T_{\gamma_i}(F)$) and let $\L_\gamma=\prod_i\L_{i}\subset T_\gamma(F)$. It is clear that $\L_\gamma\backslash X^L_\gamma$ is a product of $\L_i\backslash X^{L_i}_{\gamma_i}$. Hence by the analog of \eqref{OX} for $L$ and $L_i$, we have
\begin{equation*}
O^L_\gamma=\prod_{i\in B}O^{L_i}_{\gamma_i}
\end{equation*}

Choosing a parabolic subgroup $P\subset G$ with Levi subgroup $L$ and unipotent radical $N_P$, we have a Cartan decomposition $G(F)=P(F)G(\calO_F)=N_P(F)L(F)G(\calO_F)$. The assignment$g=nlG(\calO_F)\in G(F)/G(\calO_F)\mapsto l\in L(F)/L(\calO_F)$ gives a well-defined map $p:G(F)/G(\calO_F)\to L(F)/L(\calO_F)$ and restricts to a map $p_\gamma:X_\gamma\to X^L_\gamma$.

Let $\rho^L_G(\gamma)=\sum_{\{i,j\}\subset B,i\neq j}\val_F(\Res(f_i,f_j))$, where $f_i$ is the characteristic polynomial of $\gamma_i$.
 
\begin{lemma}\label{l:Levi} 
The fibers of the map $p_\gamma:X_\gamma\to X^L_\gamma$ all have cardinality $q^{\rho^L_G(\gamma)}$.
\end{lemma}
\begin{proof}
The choice of the parabolic $P$ is the same as an ordering of the set $B(\gamma)$, which thus allows us to identify $B(\gamma)$ with the set $\{1,2,\cdots,b\}$ for some integer $b\geq1$. This parabolic $P$ also gives a flag of the vector space $V=F^n$:
\begin{equation*}
V_{\leq 1}\subset V_{\leq 2}\subset\cdots V_{\leq b}=V
\end{equation*}
such that $\Gr_iV:=V_{\leq i}/V_{\leq i-1}$ is a free $E_i$-module of rank one. A point $x\in X^L_\gamma$ is the same as a choice of lattices $U_i\subset \Gr_iV$, one for each $i$, such that $U_i$ is stable under $\gamma_i$. The fiber of $p_\gamma$ over $x=(U_i)$ is in bijection with the set of lattices $U\subset V$, stable under $\gamma$, such that $U_{\leq i}/U_{\leq i-1}=U_i$ (where $U_{\leq i}=U\cap V_{\leq i}$) for each $i=1,2,\cdots, b$.

By induction on $b=\#B$, one reduces to the case where $B=\{1,2\}$. For each $\calO_F$-linear map $\phi:U_2\to V_1/U_1$ which satisfies $\phi\gamma_2=\gamma_1\phi$, we get a lattice $U\subset V$ spanned by $U_1$ and $\{u_2+\phi(u_2)\mod U_1|u_2\in U_2\}$. It is easy to see that such $\phi$'s are in bijection with the fiber $p_\gamma^{-1}(U_1,U_2)$. Therefore we only need to compute the number of such $\phi$'s. However, such a $\phi$ is the same as an element $\psi\in U^\vee_2\otimes_{\calO_F}(V_1/U_1)=(V^\vee_2\otimes_F V_1)/(U^\vee_2\otimes_{\calO_F}U_1)$ which is killed by the operator $\ep=\gamma^\vee_2\otimes\id-\id\otimes\gamma_1\in\End(V^\vee_1\otimes V_2)$, i.e., $\psi\in\ep^{-1}(U^\vee_2\otimes_{\calO_F}U_1)/(U^\vee_2\otimes_{\calO_F}U_1)$. The definition of the resultant implies that $\rho^L_G(\gamma)=\val_F\det(\ep)$, which is also the length of $\ep^{-1}(U^\vee_2\otimes_{\calO_F}U_1)/(U^\vee_2\otimes_{\calO_F}U_1)$ as an $\calO_F$-module. Therefore, the number of $\psi$'s, which is the same as the cardinality of $p_\gamma^{-1}(U_1,U_2)$, is $q^{\rho^L_G(\gamma)}$.
\end{proof}
When $F$ is a local function field, one can appeal to (an positive characteristic analogue of) a geometric result of Kazhdan and Lusztig \cite[\S5,Proposition 1]{KL} for affine Springer fibers.

Applying Lemma \ref{l:Levi} to the Levi $L\subset G$ which corresponds to the factorization of $f(X)$ into irreducible polynomials, we get
\begin{cor} We have
\begin{equation*}
O_\gamma=q^{\rho(\gamma)}\prod_{i\in B(\gamma)}O^{L_i}_{\gamma_i}.
\end{equation*}
\end{cor}
Moreover, each $O^{L_i}_{\gamma_i}$ may be computed by changing $F$ to $F_i=F\otimes_kk_i$ and work with $\GL(n_i/d_i,F_i)$ (to which $\gamma_i$ belongs after conjugation) instead of $\GL(n_i,F)$. Therefore we have reduced the proof of Theorem \ref{th:main} to the case $\gamma$ is elliptic and the residue field of $R$ is the same as that of $\calO_F$. From now on we will restrict ourselves to this situation. We denote the residue field of $E$ by $\tilk$ and let $r=[\tilk:k]$.

\subsection{The lower bound} We a very coarse lower bound
\begin{equation*}
\#\Hilb^i_R\geq1.
\end{equation*}
Therefore, when $r\leq\delta$, by \eqref{hilb}, we have
\begin{equation*}
O_\gamma\geq\sum_{i=0}^{\delta-r-1}(q^{\delta-i}-q^{\delta-r-i})+q^{r}+\sum_{i=\delta-r+1}^{\delta-1}(q^{\delta-i}+1)+1=q^{\delta-r+1}(q^{r-1}+\cdots+1)+r.
\end{equation*}
When $r>\delta$, a similar calculation shows that $O_\gamma\geq q^\delta+\cdots+q+\delta+1$. This shows the lower bound $O_\gamma\geq N_{\delta,r}(q)$.

\subsection{The upper bound} To obtain an upper bound for $O_\gamma$, we only need to give an upper bound for $\#\Hilb^i_R$ for each $0\leq i\leq \delta$ because all coefficients in \eqref{hilb} are positive. Since $R$ is a quotient of a formal power series ring $\calO_F[[X]]$, we have the naive estimate
\begin{equation*}
\#\Hilb^j_R\leq\#\Hilb^j_{\calO_F[[X]]}.
\end{equation*}

\begin{prop}\label{p:2dim} We have
\begin{equation*}
\#\Hilb^j_{\calO_F[[X]]}=\sum_{|\l|=j}q^{j-\ell(\l)}.
\end{equation*}
Here the sum is over partitions $\l$ of size $j$, and $\ell(\l)$ is the number of parts of $\l$.
\end{prop}
\begin{proof}
The leading term of an element $f(X)\in\calO_F[[X]]$ is the monomial appearing in $f(X)$ with lowest degree (together with its coefficient). For each ideal $I\subset \calO_F[[X]]$, let $I^\natural$ be the ideal generated by the leading terms of elements in $I$.

Each $I^\natural$ is of the form $I_\mu=(\pi^{\mu_0},\pi^{\mu_1}X,\pi^{\mu_2}X^2,\cdots)$ for a sequence $\mu_0\geq\mu_1\geq\cdots$. We view this sequence as a partition $\mu=(\mu_i)_{i\geq0}$. The size of the partition is $|\mu|=\sum\mu_i$, which is the length of $\calO_F[[X]]/I_\mu$.

We first show that $I^\natural$ has the same length as $I$, provided that the latter is finite. In fact, let $I_m=(I,X^m)\subset\calO_F[[X]]$ (with the convention that $I_0=\calO_F[[X]]$), then we have
\begin{equation}\label{fil dim}
\leng(\calO_F[[X]]/I)=\sum_{m\geq0}\leng(I_m/I_{m+1}).
\end{equation}
We also have isomorphisms $\calO_F/\pi^{\mu_m}\calO_F\isom I^\natural_m/I^\natural_{m+1}$ and $\calO_F/\pi^{\mu_m}\calO_F\isom I_m/I_{m+1}$ both given by $a\mapsto aX^m$. Therefore \eqref{fil dim} applied to both $I^\natural$ and $I$ confirms that $I^\natural$ and $I$ have the same colength.

Now we count how many $I$ satisfies $I^\natural=I_\mu$ for a fixed partition $\mu$. Suppose $I\cap(X^{m+1})$ is determined and we would like to determine $I\cap(X^m)$. Clearly, $I\cap(X^m)$ is generated by $I\cap(X^m)$ and an element $f_m(X)\in\calO_F[[X]]$ of the form
\begin{equation*}
f_m(X)=\pi^{\mu_m}X^m+\sum_{i>m}a_{i}X^i
\end{equation*}
where $a_{i}\in\calO_F/\pi^{\mu_i}\calO_F$. There are constraints on the coefficients $a_{i}$. Since $I^\natural=I_\mu$, there is an element $f_{m+1}(X)\in I\cap(X^{m+1})$ of the form$f_{m+1}(X)\equiv\pi^{\mu_{m+1}}X^{m+1}\mod X^{m+2}$. Then $Xf_m(X)-\pi^{\mu_m-\mu_{m+1}}f_{m+1}(X)\in I\cap(X^{m+2})$. Comparing the coefficient of $X^{m+2}$, we find that $a_{m+1}$ is already determined modulo $\pi^{\mu_{m+2}}$. Hence the number of choices for $a_{m+1}$ is $q^{\mu_{m+1}-\mu_{m+2}}$. Continuing like this, once $a_{m+1},\cdots,a_{i-1}$ is determined, the number of choices for $a_{i}$ is $q^{\mu_i-\mu_{i+1}}$. Therefore, the number of choices for $f_m(X)$, hence $I\cap(X^m)$ is
\begin{equation*}
\prod_{i>m}q^{\mu_i-\mu_{i+1}}=q^{\mu_{m+1}}.
\end{equation*}

Continuing this argument inductively for $m$, we see that the total number of $I$ with the same $I^\natural=I_\mu$ is
\begin{equation*}
\prod_{m\geq0}q^{\mu_{m+1}}=q^{\mu_1+\mu_2+\cdots}=q^{|\mu|-\mu_0}.
\end{equation*}
If we take the transposition $\l$ of $\mu$, then $|\mu|-\mu_0=|\l|-\ell(\l)$. Summing over all possible partitions $\l$, we get the desired formula.
\end{proof}

When $F$ is a local function field, Proposition \ref{p:2dim} can also be deduced from the cellular decomposition of the punctual Hilbert scheme for $\AA^2$, see \cite[Theorem 1.1(iv)]{ES}.

We can now finish the proof of the upper bound in Theorem \ref{th:main}. Since $O_\gamma$ is less than or equal to the quantity on the right side of \eqref{hilb} if $\Hilb^j_R$ is replaced by $\Hilb^j_{\calO_F[[X]]}$, we get from \eqref{Lef} that
\begin{eqnarray}\label{oleq}
\notag O_\gamma&\leq&\#\Hilb^{\delta}_{\calO_F[[X]]}-\#\Hilb^{\delta-r}_{\calO_F[[X]]}+\sum_{j=0}^{\delta-1}(\#\Hilb^j_{\calO_F[[X]]}-\#\Hilb^{j-r}_{\calO_F[[X]]})(1+q^{\delta-j})\\
 &=&\sum_{j=0}^{\delta}q^{\delta-j}(\#\Hilb^j_{\calO_F[[X]]}-\#\Hilb^{j-r}_{\calO_F[[X]]})+\sum_{j=\delta-r}^{\delta-1}\#\Hilb^j_{\calO_F[[X]]}.
\end{eqnarray}
Using Proposition \ref{p:2dim} we have
\begin{equation}\label{ht}
\#\Hilb^j_{\calO_F[[X]]}-\#\Hilb^{j-r}_{\calO_F[[X]]}=\sum_{|\l|=j}q^{|\l|-\ell(\l)}-\sum_{|\l|=j-r}q^{|\l|-\ell(\l)}.
\end{equation}
For a partition $\l$, we assign a new partition $\mu$ by adding $r$ 1's to $\l$, so that $m_1(\mu)\geq r$ and $|\mu|-\ell(\mu)=|\l|-\ell(\l)$. This sets up a bijection between $\{\l||\l|=j-r\}$ and $\{\mu||\mu|=j,m_1(\mu)\geq r\}$ preserving the function $\l\mapsto |\l|-\ell(\l)$. We can then rewrite \eqref{ht} as
\begin{equation*}
\#\Hilb^j_{\calO_F[[X]]}-\#\Hilb^{j-r}_{\calO_F[[X]]}=\sum_{|\l|=j,m_1(j)<r}q^{j-\ell(\l)}.
\end{equation*}
Plugging this into \eqref{oleq} and use Proposition \ref{p:2dim} again we get the desired upper bound $O_\gamma\leq M_{\delta,r}(q)$.

\begin{remark}\label{r:XR} In Example \ref{ex:n}, the order $R$ is not generated by one element over $\calO_F$ when $n\geq3$, so our estimate in Theorem \ref{th:main} does not apply. By \eqref{P_R} we have
\begin{equation*}
P_R(1)=\sum_{c=0}^{n-1}\mbox{coeff. of $t^{n-c-1}$ in }\frac{(1-t)^n}{(1-t)(1-qt)\cdots(1-q^{c+1}t)}.
\end{equation*}
It is easy to see that the leading term is $q^{n^2/4}$ for $n$ even and $2q^{(n^2-1)/4}$ for $n$ odd, which has much bigger degree than $2\delta=2(n-1)$.

Now suppose $F$ is a local function field. We give a geometric description of $X_R$ in this case. For each subset $I\subset\{1,2,\cdots,n\}$ and a number $0\leq m\leq\# I-1$ ($m=0$ if $I=\varnothing$), we call a {\em cell of type $(I,m)$ with center $\l\in\ZZ^n$} the subset $C(I,m,\l)\subset \ZZ^n$ consisting of all $n$-tuples $(\mu_1,\cdots,\mu_n)\in\ZZ^n$ such that
\begin{itemize}
\item $\mu_i=\l_i$ if $i\notin I$;
\item $\mu_i=\l_i$ or $\mu_i=\l_i+1$ if $i\in I$;
\item $\sum_i\mu_i=\sum_i\l_i+m$. 
\end{itemize}
This is a finite subset of $\ZZ^n$ of cardinality $\binom{\#I}{m}$. Let $\calC$ be the set of all such triples $(I,m,\l)$. We introduce a partial ordering on $\calC$ by declaring $(I,m,\l)\leq (I',m',\l')$ if and only if $C(I,m,\l)\subset C(I',m',\l')$. Equivalently, $(I,m,\l)\leq (I',m',\l')$ if and only if 
\begin{itemize}
\item $I\subset I'$, $m\leq m'$;
\item $\l-\l'$ consists of 0's and 1's, the 1's only occur in coordinates $I'-I$ and there are exactly $m'-m$ of them. 
\end{itemize}
For each $(I,m,\l)$, we assign to it the variety $\Gr(k^I,m)$, the Grassmannian of $m$-planes in the vector space $k^I$. For $(I,m,\l)\leq(I',m',\l')$, we have a natural embedding $\Gr(k^I,m)\incl\Gr(k^{I'},m')$ given by $V\mapsto V\oplus k^J$ where $J\subset I'-I$ is the subset of $i$ such that $\l_i=\l_i'+1$. Then, on the level of the reduced (ind-)schemes, we have 
\begin{equation*}
X_R=\varinjlim_{(I,m,\l)\in\calC}\Gr(k^I,m).
\end{equation*}
The action of $\L=\ZZ^n$ on $X_R$ is induced from its action on the index set $\calC$ by $\mu\cdot(I,m,\l)=(I,m,\l+\mu)$. In particular, there are $n-1$ irreducible components of $X_R$, isomorphic to $\Gr(k^n,m)$ where $1\leq m\leq n-1$.

We may visualize $\L\backslash X_R$ in the following way. Draw a cube $[0,1]^{n-1}$ and subdivide it into $n-1$ polytopes using the slices $S_i=\{(x_1,\cdots,x_{n-1})\in[0,1]^{n-1}|\sum_jx_j=i\}$, for $i=1,\cdots,n-1$. Fill in the polytope between $S_{i-1}$ and $S_i$ by the Grassmannian $\Gr(k^n,i)$. Then glue each pair of opposite faces of the cube together.

For example, when $n=2$, we may draw $\L\backslash X_R$ as
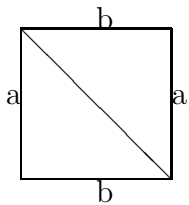
\begin{center}
\setlength{\unitlength}{1cm}
\begin{picture}(4,3.5)
\put(1,1){\line(1,0){2}}
\put(1,1){\line(0,1){2}}
\put(1,3){\line(1,0){2}}
\put(1,3){\line(1,-1){2}}
\put(3,1){\line(0,1){2}}
\put(0.8,2){a}
\put(3,2){a}
\put(2,0.7){b}
\put(2,3){b}
\end{picture}
\end{center}
Here, each triangle stands for a copy of $\PP^2$, and each segment stands for $\PP^1$. The segments with the same labels are identified. Therefore, $\L\backslash X_R$ contains two copies of $\Gm^2$, three copies of $\Gm$ and a point, totaling $2(q-1)^2+3(q-1)+1=2q^2-q$ points, which coincides with $P_R(1)$.

\end{remark}


\noindent\textbf{Acknowledgement} I would like to thank Brian Conrad, Herv\'e Jacquet and Chia-Fu Yu for stimulating discussions.


\end{document}